\documentclass[10pt,reqno]{amsart}
\usepackage{amssymb,amscd}   

\newtheorem{theorem}{Theorem}[section]
\newtheorem{proposition}[theorem]{Proposition} 
\newtheorem{corollary}[theorem]{Corollary}
\newtheorem{lemma}[theorem]{Lemma}
\newtheorem{remark}[theorem]{Remark}
\newtheorem{remarks}[theorem]{Remarks}

\numberwithin{equation}{section}

\begin{document}
\title[Free Mutual Information and Orbital Free Entropy]{Remarks on Free Mutual Information \\ and Orbital Free Entropy}
\author[M.~Izumi]
{Masaki Izumi$\,^{1}$}
\address{(M.I.) Department of Mathematics, Graduate School of Science, Kyoto
University, Sakyo-ku, Kyoto 606-8502, Japan}
\email{izumi@math.kyoto-u.ac.jp}
\thanks{$^{1}$Supported in part by Grant-in-Aid for Scientific Research (B) 22340032.}
\author[Y.~Ueda]
{Yoshimichi Ueda$\,^{2}$}
\address{(Y.U.) Graduate School of Mathematics, 
Kyushu University, 
Fukuoka, 819-0395, Japan
}
\email{ueda@math.kyushu-u.ac.jp}
\thanks{$^{2}$Supported in part by Grant-in-Aid for Scientific Research (C) 24540214.}

\maketitle

\begin{abstract} The present notes provide a proof of $i^*(\mathbb{C}P+\mathbb{C}(I-P)\,;\mathbb{C}Q+\mathbb{C}(I-Q)) = -\chi_\mathrm{orb}(P,Q)$ for {\it any} pair of projections $P,Q$ with $\tau(P)=\tau(Q)=1/2$. The proof includes new extra observations, such as a subordination result in terms of Loewner equations. A study of the general case is also given.     
\end{abstract}

\allowdisplaybreaks{

\section{Introduction} 
There are two quantities which play a r\^ole of mutual information in free probability; one is the so-called free mutual information $i^*$ introduced by Voiculescu \cite{Voiculescu:AdvMath99} in the late 90s and the other is the orbital free entropy $\chi_\mathrm{orb}$ due to Hiai, Miyamoto and the second-named author \cite{HiaiMiyamotoUeda:IJM09},\cite{Ueda:Preprint12} (and its new approaches $\tilde{\chi}_\mathrm{orb}$, etc.~due to Biane and Dabrowski \cite{BianeDabrowski:AdvMath13}). These quantities have many properties in common, but no general relationship between them has been established so far. Any question about $i^*$ and/or $\chi_\mathrm{orb}$ {\it for two projections} is known to be a `commutative one' in essence, that is, can essentially be handled within classical analysis (see \cite[\S12]{Voiculescu:AdvMath99} and \cite{HiaiPetz:ACTA06}), and a heuristic argument in \cite{HiaiUeda:AIHP09} supports that the identity $i^* = -\chi_\mathrm{orb}$ holds at least for two projections. Hence the question of $i^*=-\chi_\mathrm{orb}$ for two projections seems most tractable in the direction, and can be regarded as a counterpart of the single variable unification between two approaches $\chi$ and $\chi^*$ of free entropy, which was already established by Voiculescu (see \cite{Voiculescu:Survey}). Recently Collins and Kemp \cite{CollinsKemp:Preprint12} gave a proof of $i^* = -\chi_\mathrm{orb}$ for two projections with $\tau(P) = \tau(Q) = 1/2$ under a rather restricted assumption, along the lines of the above-mentioned heuristic argument. Here we give an improved assertion of their result (i.e., completion of the analysis when $\tau(P)=\tau(Q)=1/2$) with a rather short and completely independent proof. Originally the first-named author observed important ideas after the appearance of \cite{HiaiUeda:AIHP09} as a preprint, and then we prepared an essential part of the present short notes some years ago (see e.g.~the introduction of \cite{Ueda:Preprint12}). Although the main theorem of the present notes is still an assertion about only the case of $\tau(P) = \tau(Q) = 1/2$, a large part of its proof deals with general two projections and involves new extra observations which also enable us to give a partial result in the case of general trace values $\tau(P),\tau(Q)$. Hence the present notes may have some degree of positive significance for future studies in the direction. We should also emphasize that the attempts are important as positive evidence for the conjecture that $i^* = -\chi_\mathrm{orb}$ should hold for general random multivariables, though they have no direct connection with the unification conjecture for free entropy.   

Throughout the present notes, let $(\mathcal{M},\tau)$ denote a sufficiently large, tracial $W^*$-probability space so that all the non-commutative random variables that we will deal with live in $(\mathcal{M},\tau)$. The operator norm is denoted by $\Vert-\Vert_\infty$. Let $S_t$, $t \in [0,\infty)$, be a free additive Brownian motion in $(\mathcal{M},\tau)$ (with $S_0 = 0$). A free unitary multiplicative Brownian motion $U_t$, $t \in [0,\infty)$, with $U_0 = I$ introduced by Biane \cite{Biane:FieldsInstituteComm97} is a non-commutative process consisting of unitary random variables determined by the free stochastic differential equation (free SDE for short) $dU_t = \sqrt{-1}\,dS_t\,U_t - (1/2)U_t\,dt$, $U_0=I$. For given two projections $P, Q$ in $\mathcal{M}$ that are freely independent of $\{U_t\}_{t\geq0}$ the main objective here is to investigate the so-called liberation process $t \in [0,\infty) \mapsto (U_t (\mathbb{C}P+\mathbb{C}(I-P))U_t^*,\mathbb{C}Q+\mathbb{C}(I-Q))$ introduced by Voiculescu \cite{Voiculescu:AdvMath99} in relation with $i^*$ and $\chi_{\mathrm{orb}}$. It is known that the liberation process can be understood by looking at the process of self-adjoint random variables $X_t := Q U_t P U_t^* Q$. Thus we mainly investigate the process $X_t$ in what follows. One can easily derive the free SDE $dX_t = \Xi_t\,\sharp\,dS_t + Y_t\,dt$, where $\Xi_t := \sqrt{-1}(Q\otimes U_t P U_t^* Q - Q U_t P U_t^*\otimes Q)$ and $Y_t := \tau(P)Q - X_t$. See \cite{BianeSpeicher:PTRF98} for the definitions and the notations concerning free SDE's such as $\sharp$-operation. Note that $U_t$ is operator-norm continuous in $t$ by \cite[Lemma 8]{Biane:FieldsInstituteComm97}, and so are $X_t$, $\Xi_t$ and $Y_t$ too.  

\section{Free SDE of $(zI-X_t)^{-1}$ and Cauchy transform of $X_t$}

Several ways to investigate the free SDE of the resolvent process $R(t,z) := (zI-X_t)^{-1}$ and the Cauchy transform of $X_t$ have already been available, see e.g.~\cite[\S6--7]{Demni:JTheorProbab08},\cite[\S\S3.2]{Kargin:JTheorProbab11},\cite[\S\S3.1]{Dabrowski:ArXiv10} and \cite{CollinsKemp:Preprint12}. However, we do give, for the reader's convenience, a simple proof of their explicit formulas by simple algebraic manipulations based on three naturally expected facts -- (i) the free It\^{o} formula, (ii) the resolvent process becomes again a `free It\^{o} process' and (iii) every `free It\^{o} process' has a unique `Doob--Meyer decomposition'. In fact, the essential part of our proof will be done in several lines. The above (i) and (ii) were perfectly provided by Biane and Speicher \cite{BianeSpeicher:PTRF98}, while the above (iii) is the latter half part of Proposition \ref{P-2.2} below. The proposition (with its lemma) is probably a folklore.  

\begin{lemma}\label{L-2.1} Let $\{\mathcal{M}_t\}_{t\geq0}$ be an increasing filtration of von Neumann subalgebras of $\mathcal{M}$, and let $t \in [0,\infty) \mapsto K_t$ be a weakly measurable process such that $K_t \in \mathcal{M}_t$ and $\sup_{0\leq s \leq t}\Vert K_s\Vert_\infty < +\infty$ for all $t \geq 0$. If $t \in [0,\infty) \mapsto L_t := \int_0^t K_s\,ds$ defines a martingale adapted to $\{\mathcal{M}_t\}_{t\geq0}$, then $L_t = 0$ for all $t \geq 0$.   
\end{lemma}
\begin{proof} 
Since $L_t$ is a martingale, one has, for any division $0 = : t_0 < t_1 < \cdots < t_n := t$, 
\begin{align*} 
\tau(L_t^* L_t) = \sum_{i=1}^n \tau((L_{t_i}-L_{t_{i-1}})^*(L_{t_i}-L_{t_{i-1}}))
\leq t \big(\sup_{0\leq s \leq t}\Vert K_s\Vert_\infty\big)^2 \sup_{1\leq i\leq n} (t_i - t_{i-1}). 
\end{align*}
It follows that $L_t = 0$, since $\sup_{1\leq i\leq n} (t_i - t_{i-1})$ can arbitrarily be small. 
\end{proof} 

\begin{proposition}\label{P-2.2} Let $\{\mathcal{M}_t\}_{t\geq0}$ be as in Lemma \ref{L-2.1} such that $S_t \in \mathcal{M}_t$ for every $t \geq 0$. Let $t \in [0,\infty) \mapsto \Phi_t, \Phi'_t \in \mathcal{M}\otimes_{\mathrm{alg}}\mathcal{M}$ be operator-norm continuous biprocesses adapted to $\{\mathcal{M}_t\}_{t\geq 0}$ and $t \in [0,\infty) \mapsto K_t, K'_t \in \mathcal{M}$ be weakly measurable processes such that $\sup_{0\leq s\leq t}\Vert K_s\Vert_\infty < +\infty$ for every $t \geq 0$ and the same holds for $K'_t$. Then both $\Phi\,\mathbf{1}_{[0,t]}$ and $\Phi'\,\mathbf{1}_{[0,t]}$ fall in $\mathcal{B}_\infty^a$ {\rm(}see {\rm\cite[\S\S2.1]{BianeSpeicher:PTRF98})} for every $t\geq0$, and hence we have two free stochastic integrals $\int_0^t \Phi_s\,\sharp\,dS_s + \int_0^t K_s\,ds$ and $\int_0^t \Phi'_s\,\sharp\,dS_s + \int_0^t K'_s\,ds$ as in {\rm\cite[\S\S4.3] {BianeSpeicher:PTRF98}} for every $t \geq 0$. If those free stochastic integrals define the same process, then $\Phi = \Phi'$ holds and $K_t=K'_t$ does almost surely in $t$. 
\end{proposition}
\begin{proof} The first part is trivial; hence left to the reader. One has $\int_0^t (\Phi'_s - \Phi_s)\,\sharp\,dS_s = \int_0^t (K_s - K'_s)\,ds$, which must be zero by Lemma \ref{L-2.1} and \cite[Proposition 3.2.3]{BianeSpeicher:PTRF98}. Hence $\int_0^t \Phi_s\,\sharp\,dS_s = \int_0^t \Phi'_s\,\sharp\,dS_s$ and $\int_0^t K_s\,ds = \int_0^t K'_s\,ds$ hold for every $t>0$. The It\^{o} isometry \cite[\S\S3.1]{BianeSpeicher:PTRF98} immediately shows that $\Phi\,\mathbf{1}_{[0,t]} = \Phi'\,\mathbf{1}_{[0,t]}$ holds in $\mathcal{B}_2^a$ for every $t>0$, and hence $\Phi = \Phi'$ holds. We may and do assume that $\mathcal{M}$ has separable predual (with replacing it by its von Neumann subalgebra if necessary); thus one can choose a dense countable subset $\{\varphi_n\}_{n\in\mathbb{N}}$ of the predual of $\mathcal{M}$. One has $\int_{t_1}^{t_2} \varphi_n(K_s - K'_s)\,ds = 0$ for every $0 \leq t_1 < t_2 \lneqq \infty$ and $n \in \mathbb{N}$, which immediately implies that $K_t = K'_t$ holds almost surely in $t$.    
\end{proof}  

One can choose, for each $z \in \mathbb{C}^+:= \{z \in \mathbb{C}\,|\,\mathrm{Im}z > 0\}$, a rapidly decreasing function $f_z$ on $\mathbb{R}$ which coincides with $x \mapsto (z-x)^{-1}$ on a neighborhood of $[0,1]$, and thus
$dR(t,z) = d(f_z(X_t)) = (\partial f_z(X_t)\,\sharp\,\Xi_t)\,\sharp\,dS_t + 
(\partial f_z(X_t)\,\sharp\,Y_t + 1/2\Delta_{\Xi_t}f_z(X_t))\,dt$ holds by \cite[Proposition 4.3.4]{BianeSpeicher:PTRF98}. Here we do not recall the definitions of $\partial f_z(X_t)\,\sharp\,\Xi_t$, $\partial f_z(X_t)\,\sharp\,Y_t$ and $\Delta_{\Xi_t}f_z(X_t)$ (those can be found in \cite[\S\S4.3]{BianeSpeicher:PTRF98}, and remark that $\Vert \Delta_U f(X)\Vert_\infty$ can be estimated by $\mathcal{I}_2(f)\Vert U\Vert_\infty^2$ in the same way as in the discussion following \cite[Definition 4.1.1]{BianeSpeicher:PTRF98}). Here we  need only the following trivial fact: 
\begin{equation}\label{Eq-2.1} 
\sup\{\Vert \partial f_z(X_t)\,\sharp\,Y_t \Vert_\infty + 
\Vert\Delta_{\Xi_t}f_z(X_t)\Vert_\infty\,|\,t \geq 0\} < +\infty. 
\end{equation}
Write $M_t := \int_0^t(\partial f_z(X_s)\,\sharp\,\Xi_s)\,\sharp\,dS_s$, $Z_t := \partial f_z(X_t)\,\sharp\,Y_t + (1/2)\Delta_{\Xi_t}f_z(X_t)$ and $N_t := \int_0^t\Xi_s\,\sharp\,dS_s$ for short, and let $z \in \mathbb{C}^+$ be arbitrarily fixed. We have  
\begin{align*} 
0 
&= d\big(R(t,z)(zI-X_t)\big) 
= dR(t,z)\cdot(zI-X_t) + R(t,z)\cdot d(zI-X_t) - dM_t\cdot dN_t \\
&= dM_t\cdot(zI-X_t) + Z_t (zI-X_t)\,dt - R(t,z)\cdot dN_t - R(t,z)\cdot Y_t\,dt - dM_t\cdot dN_t,  
\end{align*}
and hence 
\begin{align*}
dM_t\cdot(zI-X_t) - R(t,z)\cdot dN_t = 
R(t,z)Y_t\,dt - Z_t (zI-X_t)\,dt + dM_t\cdot dN_t.
\end{align*} 
This formal computation can easily be justified by the rigorous formulas in \cite[\S\S4.1]{BianeSpeicher:PTRF98}. 
Note that $dM_t\cdot dN_t = \langle\langle \partial f_z(X_t)\,\sharp\,\Xi_t, \Xi_t\rangle\rangle\,dt$ by the free It\^{o} formula (see \cite[Definition 4.1.1]{BianeSpeicher:PTRF98} for the precise definition of $\langle\langle-,-\rangle\rangle$). Therefore, Proposition \ref{P-2.2} (which can be used thanks to \eqref{Eq-2.1}) shows that 
\begin{equation*}
\begin{aligned}
dM_t 
&= R(t,z)\cdot dN_t\cdot R(t,z) = \big((R(t,z)\otimes R(t,z))\,\sharp\,\Xi_t\big)\,\sharp\,dS_t, \\
Z_t\,dt 
&= R(t,z) Y_t R(t,z)\,dt + R(t,z)\cdot dN_t\cdot R(t,z)\cdot dN_t\cdot R(t,z). 
\end{aligned}
\end{equation*} 
It is easy to see, by the free It\^{o} formula again, that 
$$ 
dN_t\cdot R(t,z) \cdot dN_t 
= 
\big(-2\tau(X_t R(t,z))X_t + \tau(Q R(t,z))X_t + \tau(X_t R(t,z))Q\big)\,dt,
$$
and hence (the first part of) the next proposition follows.  

\begin{proposition}\label{P-2.3} For every $z \in \mathbb{C}^+$ the resolvent process $R(t,z) := (zI-X_t)^{-1}$ satisfies{\rm:}  
$$
dR(t,z) = \big((R(t,z)\otimes R(t,z))\,\sharp\,\Xi_t\big)\,\sharp\,dS_t + Z(t,z)\,dt 
$$ 
with 
\begin{equation}\label{Eq-2.2}
\begin{aligned} 
Z(t,z) =\ &\tau(P)R(t,z)QR(t,z) - R(t,z)X_t R(t,z)  
-2\tau(X_t R(t,z)) R(t,z)X_t R(t,z) \\
&+ \tau(Q R(t,z))R(t,z)X_t R(t,z) 
+\tau(X_t R(t,z)) R(t,z)Q R(t,z). 
\end{aligned}
\end{equation} 
Moreover, the Cauchy transform $G(t,z) := \tau(R(t,z))$, $z \in \mathbb{C}^+$ satisfies the following partial differential equation {\rm(}PDE for short{\rm):} 
\begin{equation*}
\frac{\partial G}{\partial t} = \frac{\partial}{\partial z}\left[
(z^2-z)G^2 + (2-\tau(P)-\tau(Q)-z)G - \frac{(1-\tau(P))(1-\tau(Q))}{z}\right]. 
\end{equation*}
\end{proposition}    
\begin{proof} The first part has already been obtained. Hence it suffice to show the desired PDE. Remark that $Z_t = Z(t,z)$ is operator-norm continuous in $t$ thanks to the fact at the end of \S1. By the martingale property, $G(t,z) = \tau(R(t,z)) = \tau(R(0,z)) + \int_0^t \tau(Z_s)\,ds$, and hence, by \eqref{Eq-2.2}, $\frac{\partial G}{\partial t} = \tau(Z_t) = \tau(P)\tau(Q R(t,z)^2) - \tau(X_t R(t,z)^2) - 2\tau(X_t R(t,z))\tau(X_t R(t,z)^2) + \tau(QR(t,z))\tau(X_t R(t,z)^2) + \tau(X_t R(t,z))\tau(Q R(t,z)^2)$. Note that $\tau(A R(t,z)^2) = -\frac{\partial}{\partial z}\tau(AR(t,z))$ for any $A \in \mathcal{M}$. Since $R(t,z) = QR(t,z)Q + z^{-1}(I-Q)$ and $I = (zI-X_t)R(t,z) = zR(t,z) - X_t R(t,z)$, we have $\tau(Q R(t,z)) = G(t,z)-\frac{1-\tau(Q)}{z}$ and $\tau(X_t R(t,z)) = zG(t,z)-1$. These altogether imply the desired PDE.   
\end{proof}

\section{Analysis of Probability Distribution of $X_t$}

Let $\nu_t$ be the probability distribution of $X_t$, i.e., a unique probability measure on $[0,1]$ determined by $G(t,z) = \int_{[0,1]} \frac{1}{z-x}\,\nu_t(dx)$, for $z \in \mathbb{C}^+$. Define $c_0(t) := \tau((I-U_t P U_t^*)\wedge(I-Q) + (I-U_t P U_t^*)\wedge Q + U_t P U_t^* \wedge(I-Q))$, $c_1(t) := \tau(U_t P U_t^* \wedge Q)$, $t \geq 0$. Several facts \cite[Corollary 1.7, Proposition 8.7, Corollary 8.6 and Lemma 12.5]{Voiculescu:AdvMath99} on liberation gradients with e.g.~\cite[(1.3)]{HiaiUeda:AIHP09} altogether show that the projections $U_t P U_t^*, Q$ are in generic position for every $t > 0$ and moreover that both $c_0(t) = 1-\min\{\tau(P),\tau(Q)\}$ and $c_1(t) = \max\{\tau(P)+\tau(Q)-1,0\}$ hold for every $t > 0$. (We will give its detailed explanation in Remark \ref{R-3.5} at the end of this section for the reader's convenience.) By a well-known fact (see e.g.~\cite[Solution 122]{Halmos:Book}) one easily sees that the functions $t \mapsto c_i(t)$ are upper semicontinuous, and hence $c_0(0) \geq c_0(+0) = 1-\min\{\tau(P),\tau(Q)\}$ and $c_1(0) \geq c_1(+0) = \max\{\tau(P)+\tau(Q)-1,0\}$.   

Set $\mu_t := \nu_t - (1-\min\{\tau(P),\tau(Q)\})\delta_0 - (\max\{\tau(P)+\tau(Q)-1,0\})\delta_1$, $t \geq 0$, which defines a positive measure on $[0,1]$, since $c_i(0) \geq c_i(+0)$, $i=0,1$. When $t > 0$, $\mu_t$ agrees with the restriction of $\nu_t$ to $(0,1)$. Moreover, $\mu_0$ agrees with the restriction of $\nu_0$ to $(0,1)$ (or equivalently, both $c_i(0) = c_i(+0)$, $i=0,2$, hold) if and only if $P,Q$ are in generic position. (See e.g.~the proof of \cite[Theorem 3.2]{HiaiPetz:ACTA06}.) Denote by $F(t,z)$ the Cauchy transform of $\mu_t$ whose domain clearly contains $\mathbb{C}\setminus[0,1]$. A tedious computation derives the following PDE from Proposition \ref{P-2.3}:     
\begin{equation}\label{Eq-3.1}
\frac{\partial F}{\partial t} = \frac{\partial}{\partial z}\left[(z^2-z)F^2 + a(z-1)F + bzF\right]
\end{equation}       
with $a := |\tau(P)-\tau(Q)|$ and $b:=|\tau(P)+\tau(Q)-1|$. 

Similarly to Geronimus's work \cite[\S30]{Geronimus:AMST62} (based upon the so-called Szeg\"{o} mapping) we transform $z \in \mathbb{C}\setminus[0,1] \mapsto \zeta \in \mathbb{D}$, the open unit disk, by $z = (2+\zeta+\zeta^{-1})/4$ or $\zeta = 2z-1 + 2\sqrt{z^2-z}$ (note that $\zeta \in \mathbb{D}$ determines the branch of $\sqrt{z^2-z}$ with a negative real value at $z=2$). Set $L(t,\zeta) := -\sqrt{z^2-z}\,F(t,z)$. Since $\frac{d\zeta}{dz} = \zeta/\sqrt{z^2-z}$, the PDE \eqref{Eq-3.1} becomes 
\begin{equation}\label{Eq-3.2}
\frac{\partial L}{\partial t} + \zeta\frac{\partial}{\partial\zeta}\left[\left(L+a\frac{1-\zeta}{1+\zeta} + b\frac{1+\zeta}{1-\zeta}\right)L\right] = 0. 
\end{equation}  
Letting $\tilde{\mu}_t(d\theta) = \mu_t(dx)$ with $x = \cos^2(\theta/2) = \frac{1}{2}(1 + \cos\theta)$, $\theta \in [0,\pi]$, we have 
\begin{align*} 
L(t,\zeta) 
&= 
\frac{1}{4}\left(\frac{1}{\zeta}-\zeta\right)\int_{[0,\pi]}\frac{1}{\frac{1}{4}\big(2+\zeta+\frac{1}{\zeta}\big)-\cos^2(\theta/2)}\,\tilde{\mu}_t(d\theta) \\
&= 
\frac{1}{4}\left(\frac{1}{\zeta}-\zeta\right)\int_{[0,\pi]}\frac{1}{\frac{1}{4}\big(2+\zeta+\frac{1}{\zeta}\big)-\frac{1}{4}(2+e^{\sqrt{-1}\theta}+e^{-\sqrt{-1}\theta})}\,\tilde{\mu}_t(d\theta) \\
&= 
\int_{[0,\pi]} \left(-1 + \frac{e^{\sqrt{-1}\theta}}{e^{\sqrt{-1}\theta}-\zeta} + \frac{e^{-\sqrt{-1}\theta}}{e^{-\sqrt{-1}\theta}-\zeta}\right)\,\tilde{\mu}_t(d\theta),  
\end{align*}
and thus the symmetrization $\hat{\mu}_t := \frac{1}{2}(\tilde{\mu}_t + (\tilde{\mu}_t\!\upharpoonright_{(0,\pi)})\circ j^{-1})$ with $j : \theta \in (0,\pi) \mapsto -\theta\in (-\pi,0)$ satisfies 
\begin{equation}\label{Eq-3.3}
L(t,\zeta) = \int_{(-\pi,\pi]}\frac{e^{\sqrt{-1}\theta}+\zeta}{e^{\sqrt{-1}\theta}-\zeta}\,\hat{\mu}_t(d\theta). 
\end{equation}

Define $H(t,\zeta) := (L(t,\zeta)+a\frac{1-\zeta}{1+\zeta}+b\frac{1+\zeta}{1-\zeta})L(t,\zeta)$, and by \eqref{Eq-3.2} we have 
\begin{equation}\label{Eq-3.4} 
\frac{\partial H}{\partial t} + \zeta\Big(2L(t,\zeta)+a\frac{1-\zeta}{1+\zeta}+b\frac{1+\zeta}{1-\zeta}\Big)\frac{\partial H}{\partial\zeta} = 0. 
\end{equation} 
As usual, let us consider the ordinary differential equations (ODE's for short) of characteristic curve $t \mapsto \big(g_t(\zeta),u_t(\zeta) := H(t,g_t(\zeta))\big)$ associated with the PDE \eqref{Eq-3.4}: 
\begin{align}    
\dot{g}_t(\zeta) &= g_t(\zeta)\left[2L(t,g_t(\zeta)) + a \frac{1-g_t(\zeta)}{1+g_t(\zeta)} + b \frac{1+g_t(\zeta)}{1-g_t(\zeta)}\right], \quad g_0(\zeta) = \zeta, \label{Eq-3.5} \\
\dot{u}_t(\zeta) &= 0, \quad u_0(\zeta) = H(0,\zeta). \label{Eq-3.6}
\end{align}
Here the dot symbol ($\dot{\,}$) denotes the differentiation in $t$. The ODE \eqref{Eq-3.5} is nothing less than the {\it radial Loewner} (or {\it L\"{o}wner--Kufarev}) {\it equation} (or more precisely radial Loewner ODE) determined by one parameter family of measures $t \mapsto 2\hat{\mu}_t + a\delta_{\pi} + b \delta_0$. Note by e.g.~\cite[(1.3)]{HiaiUeda:AIHP09} that $2\hat{\mu}_t + a\delta_{\pi} + b \delta_0$ defines a probability measure on $\mathbb{T} = (-\pi,\pi]$ for every $t\geq0$. (This follows from the fact that $U_t P U_t^*, Q$ are in generic position for every $t > 0$ as remarked before and $\hat{\mu}_t \to \hat{\mu}_0$ weakly as $t \searrow 0$.) Thus, by a standard fact, see e.g.~\cite[Theorem 4.14]{Lawler:Book}, the radial Loewner ODE \eqref{Eq-3.5} defines a unique one-parameter family of conformal transformations $g_t : \mathbb{D}_t := \{\zeta \in \mathbb{D}\,|\,T_\zeta > t\} \twoheadrightarrow \mathbb{D}$ with $g_t(0) = 0$ and $g'_t(0) = e^t$ (the prime symbol (${}'$) denotes the differentiation in $\zeta$), where $T_\zeta$, $\zeta\in\mathbb{D}$, is the supremum of all $T$ such that a solution of \eqref{Eq-3.5} exists until time $T$ in such a way that $g_t(\zeta) \in \mathbb{D}$ holds for every $t \leq T$.       
It is known, see e.g.~\cite[Remark 4.15]{Lawler:Book} again, that the inverse $f_t := g_t^{-1} : \mathbb{D} \twoheadrightarrow \mathbb{D}_t$ satisfies 
\begin{equation}\label{Eq-3.7} 
\dot{f}_t(\zeta) = -\zeta\,f'_t(\zeta)\left[\int_{(-\pi,\pi]}\frac{e^{\sqrt{-1}\theta}+\zeta}{e^{\sqrt{-1}\theta}-\zeta}\,(2\hat{\mu}_t+a\delta_{\pi}+b\delta_0)(d\theta)\right], \quad f_0(\zeta) = \zeta, 
\end{equation}
a {\it radial Loewner PDE}. The ODE \eqref{Eq-3.6} shows that $H(t,g_t(\zeta)) = u_t(\zeta) = u_0(\zeta) = H(0,\zeta)$, and hence $H(t,\zeta) = H(0,f_t(\zeta))$ holds for all $\zeta \in \mathbb{D}$. This implies that
\begin{equation}\label{Eq-3.8}
\begin{aligned} 
L(t,\zeta) = -\frac{1}{2}\left(a\frac{1-\zeta}{1+\zeta}+b\frac{1+\zeta}{1-\zeta}\right) + \frac{1}{2}\sqrt{\left(a\frac{1-\zeta}{1+\zeta}+b\frac{1+\zeta}{1-\zeta}\right)^2+4H(0,f_t(\zeta))},
\end{aligned} 
\end{equation}
where $\sqrt{-}$ is the principal branch. The discussions so far are summarized as follows. 

\begin{proposition}\label{P-3.1} Let $\nu_t$ be the probability distribution of $X_t$. Define the positive measure $\mu_t := \nu_t - (1-\min\{\tau(P),\tau(Q)\})\delta_0 - (\max\{\tau(P)+\tau(Q)-1,0\})\delta_1$, and transform it to the positive measure $\tilde{\mu}_t(d\theta) := \mu_t(dx)$ on $[0,\pi]$ by $x=\cos^2(\theta/2)$. Then $\mu_t$ coincides with the restriction of $\nu_t$ to $(0,1)$ for every $t > 0$, and moreover so does for $t=0$ {\rm(}or equivalently, $\mu_0$ has no atom at both $0$ and $1${\rm)} if and only if the given two projections $P, Q$ are in generic position.  

Set $L(t,\zeta) := \int_{(-\pi,\pi)}\frac{e^{\sqrt{-1}\theta}+\zeta}{e^{\sqrt{-1}\theta}-\zeta}\,\hat{\mu}_t(d\theta)$, $\zeta \in \mathbb{D}$, with the symmetrization $\hat{\mu}_t := \frac{1}{2}(\tilde{\mu}_t + (\tilde{\mu}_t\!\upharpoonright_{(0,\pi)})\circ j^{-1})$ with $j : \theta \in (0,\pi) \mapsto -\theta\in (-\pi,0)$. Then the unique one-parameter, subordinate family of conformal self-maps $f_t$ on $\mathbb{D}$ obtained from the radial Loewner PDE \eqref{Eq-3.7} driven by the probability measures $2\hat{\mu}_t + a\delta_{\pi} + b\delta_0$ gives the following subordination relation{\rm:}  
\begin{equation*}
\begin{aligned} 
\Big(L(t,\zeta)+a\frac{1-\zeta}{1+\zeta}+b\frac{1+\zeta}{1-\zeta}\Big)L(t,\zeta) 
= \Big(L(0,f_t(\zeta))+a\frac{1-f_t(\zeta)}{1+f_t(\zeta)}+b\frac{1+f_t(\zeta)}{1-f_t(\zeta)}\Big)L(0,f_t(\zeta))
\end{aligned} 
\end{equation*}
with $a = |\tau(P)-\tau(Q)|$ and $b = |\tau(P)+\tau(Q)-1|$.  
\end{proposition} 

The next corollary is a specialization of the above proposition.  

\begin{corollary}\label{C-3.2} Let $L(t,\zeta)$, $f_t(\zeta)$ be as in Proposition \ref{P-3.1}, set $g_t(\zeta):=f_t^{-1}(\zeta)$, and suppose that $\tau(P)=\tau(Q)=1/2$ or equivalently $a=b=0$. Then 
\begin{itemize}
\item $L(t,\zeta) = L(0,f_t(\zeta))$, that is, $L(t,\zeta)$ is subordinate to $L(s,\zeta)$ for $s < t$, 
\item $g_t(\zeta) = \zeta e^{2t L(0,\zeta)}$ and $f_t(\zeta) = \zeta e^{-2t L(t,\zeta)}$,  
\item $\mathrm{Re}L(t,\zeta) = (\log|\zeta|-\log|f_t(\zeta)|)/2t$, $t>0$ and $\zeta \in \mathbb{D}\setminus\{0\}$.
\end{itemize} 
\end{corollary}
\begin{proof} Under the assumption here the subordination relation in Proposition \ref{P-3.1} turns out to be the exact subordination $L(t,\zeta) = L(0,f_t(\zeta))$. This together with \eqref{Eq-3.5} implies that $\dot{g}_t(\zeta) = 2g_t(\zeta)L(t,g_t(\zeta)) = 2g_t(\zeta)L(0,\zeta)$ . This ODE can easily be solved as $g_t(\zeta) = \zeta e^{2t L(0,\zeta)}$, implying $\zeta = f_t(\zeta)e^{2t L(0,f_t(\zeta))} = f_t(\zeta)e^{2t L(t,\zeta)}$. The final assertion immediately follows. \end{proof} 

This allows us to prove some properties of $\hat{\mu}_t$ by analyzing $f_t(\zeta)$ and/or $g_t(\zeta)$ when $\tau(P)=\tau(Q)=1/2$, but we give a more useful observation as the next proposition. The proposition immediately follows from only \eqref{Eq-3.2} and \eqref{Eq-3.3}. This means that the proof of the main result of the present notes (Theorem \ref{T-4.3}) needs only a few pages.   

\begin{proposition}\label{P-3.3} Under the same assumption as in Corollary 3.2, $\{2\hat{\mu}_{t/2}\}_{t\geq0}$ is identical to the one-parameter semigroup of probability distributions associated with a free unitary multiplicative Brownian motion with initial distribution $2\hat{\mu}_0$. 
\end{proposition}  
\begin{proof} Since $\hat{\mu}_t$ is symmetric, we have $\psi(t,\zeta) := \int_{(-\pi,\pi]}\frac{\zeta e^{\sqrt{-1}\theta}}{1-\zeta e^{\sqrt{-1}\theta}}\,(2\hat{\mu}_{t/2})(d\theta) = L(t/2,\zeta)-1/2$, the moment generating function of the measure $2\hat{\mu}_{t/2}$. The PDE \eqref{Eq-3.2} can easily be transformed into 
\begin{equation}\label{Eq-3.9}
\dot{\psi} + \zeta(\psi+1/2)\psi'=0. 
\end{equation}
This is the PDE that the moment generating function of a free unitary multiplicative Brownian motion satisfies, see e.g.~the proof of \cite[Proposition 10.8]{Voiculescu:AdvMath99}, and hence the desired assertion follows as seen below. Let $U$ be a unitary random variable with distribution $2\hat{\mu}_0$, which is freely independent of $\{U_t\}_{t\geq0}$. Set $\tilde{\psi}(t,\zeta) := \tau((I-\zeta U_t U)^{-1}-I)$, $\zeta \in \mathbb{D}$, the moment generating function of $U_t U$. Then $\tilde{\psi}$ satisfies the same PDE \eqref{Eq-3.9}. Write $\psi(t,\zeta) = \sum_{n=1}^\infty c_n(t)\zeta^n$, $\tilde{\psi}(t,\zeta) = \sum_{n=1}^\infty \tilde{c}_n(t)\zeta^n$. Developing \eqref{Eq-3.9} into power series as above we see that both the coefficients $c_n$ and $\tilde{c}_n$ must satisfy that $\dot{f_1} = -\frac{1}{2}f_1$, $\dot{f_n} = -\frac{n}{2}f_n - \sum_{k=1}^{n-1} kf_k f_{n-k}$ ($n=2,3,\dots$) with $f_n = c_n$ or $\tilde{c}_n$. Since $\psi(0,\zeta)=\int_{(-\pi,\pi)}\frac{\zeta e^{\sqrt{-1}\theta}}{1-\zeta e^{\sqrt{-1}\theta}}\,(2\hat{\mu}_0)(d\theta) = \tilde{\psi}(0,\zeta)$, $\zeta \in \mathbb{D}$, one has $c_n(0) = \tilde{c}_n(0)$ for every $n$. Hence one can recursively show that $\int_{(-\pi,\pi]} e^{\sqrt{-1}n\theta}\,(2\hat{\mu}_{t/2})(d\theta)=c_n(t)=\tilde{c}_n(t)=\tau((U_t U)^n)$. \end{proof} 

\begin{remarks}\label{R-3.4} {\rm 
(1) The above proposition enables us to derive detailed information about $\mu_t$ from many existing results \cite{Biane:FieldsInstituteComm97},\cite[\S\S4.2]{Biane:JFA97},\cite[\S1]{Voiculescu:AdvMath99} on free unitary multiplicative Brownian motions (with the help of $S$-transform machinery, see e.g.~\cite[\S3]{Voiculescu:Saint-Flour98}) when $\tau(P)=\tau(Q)=1/2$. Moreover, the recent work \cite{Zhong:Preprint13} generalizing Biane's analysis \cite[\S\S4.2]{Biane:JFA97} gives more detailed properties of $\hat{\mu}_t$ and hence those of $\mu_t$, though we omit to collect any result in the direction here.   

(2) The above proposition also recaptures, as its  specialization, the main theorem of \cite{DemniHamdiHmidi:Preprint12}. In fact, the free Jacobi process with parameter $(\lambda,\theta) = (1,1/2)$ \cite{Demni:JTheorProbab08} is exactly our $X_t$ (viewed as a random variable in $(Q\mathcal{M}Q,\frac{1}{\tau(Q)}\tau)$) with $P=Q$ and $\tau(P)=\tau(Q)=1/2$. Hence the initial distribution $2\hat{\mu}_0$ is the unit mass at $\theta=0$, and thus the probability distribution of the free Jacobi process with parameter $(\lambda,\theta) = (1,1/2)$ is exactly that of the free unitary multiplicative Brownian motion via $x = \cos^2(\theta/2)$.        
}
\end{remarks} 

\begin{remark}\label{R-3.5} {\rm The following simple `liberation theoretic' proof of the fact that $U_t P U_t^*,Q$ are in generic position for every $t > 0$ has been available so far: By \cite[Corollary 1.7, Proposition 8.7]{Voiculescu:AdvMath99} $d^*_{U_t:\mathbb{C}}1\otimes1$ (see the notation there) exists in $L^2$ for every $t > 0$, which implies, by \cite[Corollary 8.6]{Voiculescu:AdvMath99}, that so does the liberation gradient $j(U_t(\mathbb{C}P+\mathbb{C}(I-P))U_t^*:\mathbb{C}Q+\mathbb{C}(I-Q))$. Therefore, by \cite[Lemma 12.5]{Voiculescu:AdvMath99} (together with $U_t(\mathbb{C}P+\mathbb{C}(I-P))U_t^* = \mathbb{C}U_t P U_t^*+\mathbb{C}(I-U_t P U_t^*)$) we conclude that $U_t P U_t^*, Q$ are in generic position for every $t > 0$. This argument indeed shows the following stronger result: $UPU^*, Q$ are in generic position for any unitary $U$ with finite Fisher information $F(U) <+\infty$ (\cite[Definition 8.9]{Voiculescu:AdvMath99}) which is freely independent of $P,Q$.      
}
\end{remark}      

\section{Free Mutual Information and Orbital Free Entropy}

To a given pair of projections $P,Q$ we can associate four quantities: the liberation gradient $j(\mathbb{C}P+\mathbb{C}(I-P):\mathbb{C}Q+\mathbb{C}(I-Q))$ ($=: j(P:Q)$ for short), the liberation Fisher information $\varphi^*(\mathbb{C}P+\mathbb{C}(I-P):\mathbb{C}Q+\mathbb{C}(I-Q))$ ($=: \varphi^*(P:Q)$), the mutual free information $i^*(\mathbb{C}P+\mathbb{C}(I-P):\mathbb{C}Q+\mathbb{C}(I-Q))$ ($=: i^*(P:Q)$), all of which are due to Voiculescu \cite{Voiculescu:AdvMath99}, and the orbital free entropy $\chi_\mathrm{orb}(P,Q)$ \cite{HiaiMiyamotoUeda:IJM09}. Note that $i^*(\mathbb{C}P+\mathbb{C}(I-P)\,; \mathbb{C}Q+\mathbb{C}(I-Q)) = i^*(\mathbb{C}P+\mathbb{C}(I-P):\mathbb{C}Q+\mathbb{C}(I-Q))$, see \cite[Remarks 10.2 (c)]{Voiculescu:AdvMath99}, and hence it suffices to compute the latter quantity for our purpose. According to the change of variables $\mu_t \leadsto \tilde{\mu}_t \leadsto \hat{\mu}_t$ in \S3 we need to reformulate  Voiculescu's computation of $\varphi^*(P:Q)$, \cite[\S12]{Voiculescu:AdvMath99}, as well as the previous computation of $\chi_\mathrm{orb}(P,Q)$ essentially due to Hiai and Petz \cite{HiaiPetz:ACTA06}. 

For simplicity, write $\delta := \delta_{\,\mathbb{C}P+\mathbb{C}(I-P)\,:\,\mathbb{C}Q+\mathbb{C}(I-Q)}$, the derivation associated with $\mathbb{C}P+\mathbb{C}(I-P)$ and $\mathbb{C}Q+\mathbb{C}(I-Q)$ \cite[\S\S5.3]{Voiculescu:AdvMath99}. Let $\mu$ be the restriction of the probability distribution of $QPQ$ to $(0,1)$. Note that the measure $\mu$ is not changed if $QPQ$ is replaced by $PQP$ and that $\mu$ is exactly $\frac{1}{2}\nu$ in \cite[\S12]{Voiculescu:AdvMath99}. Write $a:=|\tau(P)-\tau(Q)|$ and $b:=|\tau(P)+\tau(Q)-1|$ for simplicity. If $P,Q$ are in generic position, then by \cite[\S\S12.1--12.6]{Voiculescu:AdvMath99} one has, for $n\geq1$,  
\begin{align*} 
(\tau\otimes\tau)\circ\delta\,(PQ)^n &= 
2\,\mathrm{PV}\int\int_{(0,1)^2}x^n(x-1)\frac{1}{x-y}\,\mu\otimes\mu\,(dx,dy) \\
&\quad\quad+ (a+b)\int_{(0,1)} x^{n-1}(x-1)\,\mu(dx) + b\int_{(0,1)} x^{n-1}\,\mu(dx).
\end{align*} 
Here `$\mathrm{PV}$' is the sign of Cauchy principal value. 
With $\theta \in (0,\pi) \mapsto x=\cos^2(\theta/2) \in (0,1)$ and $\tilde{\mu}(d\theta) := \mu(dx)$ as in \S3 we have, for $n\geq1$,   
\begin{equation}\label{Eq-4.1}
\begin{aligned} 
(\tau\otimes\tau)\circ\delta\,(PQ)^n 
&= 
-2\,\mathrm{PV}\int\int_{(0,\pi)^2} \cos^{2n-1}(\alpha/2)\sin(\alpha/2)\frac{\sin\alpha}{\cos\alpha-\cos\beta}\,\tilde{\mu}\otimes\tilde{\mu}\,(d\alpha,d\beta) \\
&\quad-a\int_{(0,\pi)}\cos^{2(n-1)}(\theta/2)\sin^2(\theta/2)\,\tilde{\mu}(d\theta) + b\int_{(0,\pi)}\cos^{2n}(\theta/2)\,\tilde{\mu}(d\theta).
\end{aligned}
\end{equation}
Here we further suppose that $\mu$ has a density function $h$, i.e., $\mu(dx) = h(x)\,dx$. Set $\tilde{h}(\theta) := h(\cos^2(\theta/2))\sin(\theta/2)\cos(\theta/2)$, and thus $\tilde{\mu}(d\theta) = \tilde{h}(\theta)\,d\theta$. Then the symmetrization $\hat{\mu} := \frac{1}{2}(\tilde{\mu}+\tilde{\mu}\circ j^{-1})$ with $j : \theta \in (0,\pi) \mapsto -\theta\in (-\pi,0)$ also has a density function, that is, $\hat{\mu}(d\theta) = \hat{h}(\theta)\,d\theta$ with $\hat{h}(\theta) = (h(\cos^2(\theta/2))|\sin\theta|)/4 = (h(\cos^2(\theta/2))|\sin(\theta/2)|\cos(\theta/2))/2$, $\theta \in (-\pi,\pi)$. The Hilbert transform (or the harmonic conjugate) of $\hat{h}$ is defined by 
\begin{equation*} 
(H\hat{h})(\theta) := \frac{1}{2\pi}\,\mathrm{PV}\int\frac{\hat{h}(\phi)}{\tan((\theta-\phi)/2)}\,d\phi, \quad \theta \in \mathbb{T} = [-\pi,\pi),   
\end{equation*}     
which exists a.e., see \cite[III.C.2]{Koosis:Book}. As in \cite[\S6.7, (6.86)]{King:Book} the restriction of $H\hat{h}$ to $(0,\pi)$ can be re-written in terms of $\tilde{h}$ as follows.  
\begin{equation}\label{Eq-4.2} 
(H\hat{h})(\theta) = -\frac{\sin\theta}{2\pi}\,\mathrm{PV}\int_{(0,\pi)}\frac{\tilde{h}(\phi)}{\cos\theta-\cos\phi}\,d\phi, \quad \theta \in (0,\pi).  
\end{equation} 
Under the equivalent assumptions 
\begin{equation}\label{Eq-4.3}
\begin{aligned} 
&\text{$\hat{h} \in L^2(-\pi,\pi)$ if and only if $\tilde{h} \in L^2(0,\pi)$}, \\
&\quad\quad\quad\text{or equivalently $\int_{(0,1)}\sqrt{x(1-x)}h(x)^2\,dx < +\infty$}
\end{aligned}
\end{equation} 
the Cauchy principal value in \eqref{Eq-4.2} converges in $L^2$-norm by \cite[I.E.4]{Koosis:Book}. Define a function $\xi : (0,\pi) \rightarrow M_2(\mathbb{C})$ by 
\begin{equation*}
\xi(\theta) := \big(4\pi(H\hat{h})(\theta) - a \tan(\theta/2) + b \cot(\theta/2)\big)\begin{bmatrix} 0 & -1 \\ 1 & 0 \end{bmatrix} 
\end{equation*}  
With these preliminaries we have: 

\begin{lemma}\label{L-4.1} Assume that $P,Q$ are in generic position. If $\mu(dx) = h(x)\,dx$ such that $h$ satisfies \eqref{Eq-4.3},   
then $\xi$ gives the liberation gradient $j(\mathbb{C}P+\mathbb{C}(I-P):\mathbb{C}Q+\mathbb{C}(I-Q))$ as long as $\theta \mapsto 4\pi(H\hat{h})(\theta) - a \tan(\theta/2) + b \cot(\theta/2)$ is integrable with respect to $\tilde{\mu}$, and moreover 
\begin{equation}\label{Eq-4.4}
\begin{aligned}  
&\varphi^*(\mathbb{C}P+\mathbb{C}(I-P):\mathbb{C}Q+\mathbb{C}(I-Q)) \\
&= 
\int_{(0,\pi)} 2\big|4\pi(H\hat{h})(\theta) - a\tan(\theta/2) + b\cot(\theta/2)\big|^2\,\tilde{\mu}(d\theta) \\
&= 
\int_{(-\pi,\pi)}\big|2\pi(H(2\hat{h}))(\theta) - a\tan(\theta/2) + b\cot(\theta/2)\big|^2\,(2\hat{\mu})(d\theta)
\end{aligned}
\end{equation}
possibly to be $+\infty$ under the same integrability assumption.  
\end{lemma}     
\begin{proof}
By the computation \eqref{Eq-4.1} together with \eqref{Eq-4.2} and the hypotheses \eqref{Eq-4.3}, one can easily see that $\tau(\xi(PQ)^n) = (\tau\otimes\tau)\circ\delta\,(PQ)^n$ for $n\geq1$ (whose proof is just an translation of the proof of \cite[Proposition 12.7]{Voiculescu:AdvMath99} into the present context), and conclude $j(P:Q) = \xi$ by its definition (see \cite[Definition 5.4]{Voiculescu:AdvMath99}) under the integrability assumption. Then the first equality in \eqref{Eq-4.4} is immediate, and the second one follows from the fact that $\theta \mapsto 4\pi(H\hat{h})(\theta) - a\tan(\theta/2) + b\cot(\theta/2)$ is an odd function. \end{proof}

Keep the notations $\mu$, $\tilde{\mu}$, $\hat{\mu}$, and $a,b$ above. If $P,Q$ are in generic position, then 
\begin{equation*}
\begin{aligned} 
\chi_\mathrm{orb}(P,Q) &= 
\int\int_{(0,1)^2}\log|x-y|\,\mu\otimes\mu\,(dx,dy) \\
&\quad\quad\quad+ 
a \int_{(0,1)}\log x\,\mu(dx) + b \int_{(0,1)}\log(1-x)\,\mu(dx) + C; 
\end{aligned} 
\end{equation*}
otherwise $-\infty$, where $C$ is a unique constant determined by $\chi_\mathrm{orb}(P,Q) = 0$ when $P, Q$ are freely independent with keeping prescribed values of $\tau(P),\tau(Q)$. In particular, $C=(\log2)/2$ when $\tau(P)=\tau(Q)=1/2$. See e.g.~\cite[Lemma 1.1]{HiaiUeda:AIHP09},\cite[Lemma 2.4]{HiaiMiyamotoUeda:IJM09}. In what follows we assume that $P,Q$ are in generic position, and, in particular, $\mu((0,1)) = (1-a-b)/2$ by \cite[(1.3)]{HiaiUeda:AIHP09}. Since $|\cos\alpha-\cos\beta| = (|e^{\sqrt{-1}\alpha}-e^{\sqrt{-1}\beta}|\cdot|e^{\sqrt{-1}\alpha}-e^{-\sqrt{-1}\beta}|)/2$, with $x=\cos^2(\alpha/2)$, $y=\cos^2(\beta/2)$ we have 
\begin{align*} 
&\int\int_{(0,1)^2}\log|x-y|\,\mu\otimes\mu\,(dx,dy) 
= 
\int\int_{(0,\pi)^2}\big(\log|\cos\alpha-\cos\beta| - \log2\big)\,\tilde{\mu}\otimes\tilde{\mu}\,(d\alpha,d\beta) \\
&=
\int\int_{(0,\pi)^2}\big(
\log|e^{\sqrt{-1}\alpha}-e^{\sqrt{-1}\beta}|+\log|e^{\sqrt{-1}\alpha}-e^{-\sqrt{-1}\beta}|-2\log2\big)\,\tilde{\mu}\otimes\tilde{\mu}\,(d\alpha,d\beta) \\
&=
2\int\int_{(-\pi,\pi)^2} 
\log|e^{\sqrt{-1}\alpha}-e^{\sqrt{-1}\beta}|\,\hat{\mu}\otimes\hat{\mu}\,(d\alpha,d\beta) - \frac{\log2}{2}(1-a-b)^2. 
\end{align*}
Here we used the fact that $\mu(0,1) = \tilde{\mu}(0,\pi) = \hat{\mu}(-\pi,\pi) = (1-a-b)/2$. With $x=\cos^2(\theta/2)$ we have 
\begin{align*} 
\int_{(0,1)} \log x\,\mu(dx) &= 
2\int_{(-\pi,\pi)}\log|1+e^{\sqrt{-1}\theta}|\,\hat{\mu}(d\theta) - (1-a-b)\log2, \\
\int_{(0,1)} \log (1-x)\,\mu(dx) &= 
2\int_{(-\pi,\pi)}\log|1-e^{\sqrt{-1}\theta}|\,\hat{\mu}(d\theta) - (1-a-b)\log2. 
\end{align*} 
Therefore, we conclude:

\begin{lemma}\label{L-4.2} If $P,Q$ are in generic position, then 
\begin{align*} 
\chi_\mathrm{orb}(P,Q) 
&= 
2\Big\{
\int\int_{(-\pi,\pi)^2} 
\log|e^{\sqrt{-1}\alpha}-e^{\sqrt{-1}\beta}|\,\hat{\mu}\otimes\hat{\mu}\,(d\alpha,d\beta) \\
&\quad\quad\quad+ 
a \int_{(-\pi,\pi)}\log|1+e^{\sqrt{-1}\theta}|\,\hat{\mu}(d\theta) + 
b \int_{(-\pi,\pi)}\log|1-e^{\sqrt{-1}\theta}|\,\hat{\mu}(d\theta)
\Big\} + Z 
\end{align*}
with a universal constant $Z = Z_{\tau(P),\tau(Q)}$ depending only on $\tau(P), \tau(Q)${\rm;} otherwise $-\infty$. In particular, if $\tau(P)=\tau(Q)=1/2$, then the above formula of $\chi_\mathrm{orb}(P,Q)$ simply becomes 
\begin{equation*}
\chi_\mathrm{orb}(P,Q) = 2\int\int_{(-\pi,\pi)^2}\log|e^{\sqrt{-1}\alpha}-e^{\sqrt{-1}\beta}|\,\hat{\mu}\otimes\hat{\mu}\,(d\alpha,d\beta).
\end{equation*} 
\end{lemma}

\medskip
Let us return to the original situation; thus we use the notations in \S3. We can now reduce our question to \cite[Corollary 10.9]{Voiculescu:AdvMath99} when $\tau(P)=\tau(Q)=1/2$.    

\begin{theorem} \label{T-4.3} For any two projections $P,Q$ with $\tau(P)=\tau(Q)=1/2$ one has 
$$
i^*(\mathbb{C}P+\mathbb{C}(I-P)\,;\mathbb{C}Q+\mathbb{C}(I-Q)) = -\chi_\mathrm{orb}(P,Q) \quad \text{possibly with $+\infty = +\infty$}.
$$ 
 
\end{theorem}
\begin{proof} By Proposition \ref{P-3.3} and \cite[Corollary 1.7]{Voiculescu:AdvMath99} $2\hat{\mu}_{t/2}$ has an $L^\infty$-density $2\hat{h}(t/2,\theta)$ for every $t >0$. By \cite[Corollary 10.9]{Voiculescu:AdvMath99} we have 
\begin{equation}\label{Eq-4.5} 
\begin{aligned}
&-\int\int_{(-\pi,\pi]^2}\log|e^{\sqrt{-1}\alpha}-e^{\sqrt{-1}\beta}|\,(2\hat{\mu}_0)\otimes(2\hat{\mu}_0)\,(d\alpha,d\beta) \\
&\quad\quad\quad\quad\quad\quad=
\frac{1}{2}\int_0^{+\infty}\int_{(-\pi,\pi]}\big(2\pi H(2\hat{h}(t/2,-))(\theta)\big)^2 2\hat{h}(t/2,\theta)\,d\theta\,dt.  
\end{aligned}
\end{equation}    
By Lemma \ref{L-4.1}   
$$
\int_{(-\pi,\pi)}\big(2\pi H(2\hat{h}(t/2,-))(\theta)\big)^2 (2\hat{h}(t/2,-))(\theta)\,d\theta = \varphi^*(U_{t/2} P U_{t/2}^*:Q)
$$ 
holds for every $t>0$ so that the right-hand side of \eqref{Eq-4.6} is identical to 
$$
\frac{1}{2}\int_0^{+\infty} \varphi^*(U_{t/2} P U_{t/2}^*:Q)\,dt = \frac{1}{2}\int_0^{+\infty}\varphi^*(U_t P U_t^*:Q)\,2\,dt = 2\,i^*(P:Q) = 2\,i^*(P\,;Q).
$$ 

Assume first that $P,Q$ are in generic position. By Lemma \ref{L-4.2} the left-hand side of \eqref{Eq-4.5} is identical to $-2\,\chi_\mathrm{orb}(P,Q)$. Thus the desired identity follows. Assume next that $P,Q$ are not in generic position. By what we have done in \S3 $\hat{\mu}_0$ must have at least one atom at either $0$ or $\pi$ with weight $c_1(0)-c_1(+0) \gneqq 0$ or $c_0(0)-c_0(+0) \gneqq 0$, respectively. Thus the left-hand side of \eqref{Eq-4.5} must be $+\infty$, and therefore, so is $i^*(P\,;Q)$. By definition $\chi_\mathrm{orb}(P,Q) = -\infty$ in this case, and hence the desired identity holds as $+\infty = +\infty$.  
\end{proof} 

In closing we illustrate how one can use the subordination relation in Proposition \ref{P-3.1}. 

\begin{lemma}\label{L-4.4} If $H(t,\zeta)$ {\rm(}see \S3{\rm)} defines a function in $\zeta$ of Hardy class with exponent $3/2$ {\rm(}see {\rm\cite[IV.B.2]{Koosis:Book})} at each $t > 0$, then $i^*(\mathbb{C}P+\mathbb{C}(I-P)\,;\mathbb{C}Q+\mathbb{C}(I-Q)) = -\chi_\mathrm{orb}(P,Q)$ holds. 
\end{lemma}
\begin{proof} Let $L(t,\zeta)$ be as in \S3, and write $\tilde{L}(t,\zeta) := L(t,\zeta)+a\frac{1-\zeta}{1+\zeta}+b\frac{1+\zeta}{1-\zeta}$. By the PDE \eqref{Eq-3.1} one has $\frac{\partial \mathrm{Re}\tilde{L}}{\partial t}(\zeta) 
= \frac{\partial \mathrm{Re}L}{\partial t}(\zeta) = -\frac{\partial}{\partial \theta}\big(\mathrm{Re}\tilde{L}(\zeta)\cdot\mathrm{Im}L(\zeta) + \mathrm{Im}\tilde{L}(\zeta)\cdot\mathrm{Re}L(\zeta)\big)$ with $\zeta = re^{\sqrt{-1}\theta}$. 
Write $\Sigma_s(p,q) := 2\int\int_{(-\pi,\pi]^2} \log|1-se^{\sqrt{-1}(\alpha-\beta)}|\,p(e^{\sqrt{-1}\alpha})\,q(e^{\sqrt{-1}\beta})\,\frac{d\alpha}{2\pi}\,\frac{d\beta}{2\pi}$ for simplicity. With the Poisson kernel $P_r(\theta)$ the same trick as in the proof of \cite[Proposition 10.8]{Voiculescu:AdvMath99} shows that 
\begin{equation}\label{Eq-4.6} 
\begin{aligned}
&\Sigma_s(P_r*(\hat{\mu}_{t_2}+a\delta_\pi+b\delta_0),P_r*\hat{\mu}_{t_2}) - \Sigma_s(P_r*(\hat{\mu}_{t_1}+a\delta_\pi+b\delta_0),P_r*\hat{\mu}_{t_1}) \\
&= 
\frac{1}{2}\int_{t_1}^{t_2}\int_{(-\pi,\pi]}\mathrm{Im}\Big(2L(t,sre^{\sqrt{-1}\theta})+a\frac{1-sre^{\sqrt{-1}\theta}}{1+sre^{\sqrt{-1}\theta}}+b\frac{1+sre^{\sqrt{-1}\theta}}{1-sre^{\sqrt{-1}\theta}}\Big) \\
&\quad\times\mathrm{Im}\Big(2L(t,re^{\sqrt{-1}\theta})+a\frac{1-re^{\sqrt{-1}\theta}}{1+re^{\sqrt{-1}\theta}}+b\frac{1+re^{\sqrt{-1}\theta}}{1-re^{\sqrt{-1}\theta}}\Big)\,2\mathrm{Re}L(t,re^{\sqrt{-1}\theta})\,\frac{d\theta}{2\pi}\Big)\,dt \\ 
&\quad+\int_{t_1}^{t_2} \Big[\int_{(-\pi,\pi]} \mathrm{Im}(L+\tilde{L})(t,sre^{\sqrt{-1}\theta})\mathrm{Im}L(t,re^{\sqrt{-1}\theta})\\
&\phantom{aaaaaaaaaaaaaaaaaaaaaaaaa}\times
\mathrm{Re}\Big(a\frac{1-re^{\sqrt{-1}\theta}}{1+re^{\sqrt{-1}\theta}}+b\frac{1+re^{\sqrt{-1}\theta}}{1-re^{\sqrt{-1}\theta}}\Big)\,\frac{d\theta}{2\pi}\Big]\,dt
\end{aligned}
\end{equation} 
for every $0<t_1 < t_2 <\infty$. 

Since $\mathrm{Re}L(t,\zeta)^2 \leq |H(t,\zeta)|$, the assumption here implies that $\hat{\mu}_t$ has an $L^3$-density $\hat{h}(t,\theta)$, i.e., $\hat{\mu}_t(d\theta) = \hat{h}(t,\theta)\,d\theta$, for every $t>0$ (see \cite[p.15]{Koosis:Book}). We fix arbitrary $0 < t_1 < t_2 < +\infty$ for a while. Set $M_{t_1} := \sup_{r < 1}\Vert H(t,re^{\sqrt{-1}(-)})\Vert_{3/2} < +\infty$ by assumption, where $\Vert-\Vert_p$ denotes the usual $L^p$-norm with respect to $d\theta$ rather than $d\theta/2\pi$ following \cite{Koosis:Book}. By the subordination relation in Proposition \ref{P-3.1} with Littlewood's subordination principle (see \cite[Theorem 1.7]{Duren:Book}) one has 
\begin{equation}\label{Eq-4.7} 
\Vert\mathrm{Re}L(t,re^{\sqrt{-1}(-)})\Vert_3 \leq \Vert H(t,re^{\sqrt{-1}(-)})\Vert_{3/2}^{1/2} \leq M_{t_1}^{1/2};\,\text{hence}\hspace{0.1cm} \Vert\hat{h}(t,-)\Vert_3 \leq M_{t_1}^{1/2}/2\pi
\end{equation} 
for every $t \geq t_1$ and $0\leq r < 1$. Note that 
$$
\Big\Vert\mathrm{Im}L(t,re^{\sqrt{-1}(-)})\mathrm{Re}\Big(a\frac{1-re^{\sqrt{-1}(-)}}{1+re^{\sqrt{-1}(-)}}+b\frac{1+re^{\sqrt{-1}(-)}}{1-re^{\sqrt{-1}(-)}}\Big)\Big\Vert_{3/2} \leq \Vert H(t,re^{\sqrt{-1}(-)})\Vert_{3/2} \leq M_{t_1}
$$ 
for every $t \geq t_1$ and $0 \leq r < 1$ by the subordination relation in Proposition \ref{P-3.1} with Littlewood's subordination principle again. Using the Cauchy--Schwarz inequality (with respect to $\mathrm{Re}(\cdots)\,d\theta/2\pi\,dt$) and then the H\"{o}lder inequality (with respect to $d\theta$ and exponents $3,3/2$) with the help of M.~Riesz's theorem (see \cite[p.91]{Koosis:Book}) we see that the absolute value of the second term of the right-hand side of \eqref{Eq-4.6} is not greater than 
\begin{align*}
\Big\{\int_{t_1}^{t_2} \Big[\int_{(-\pi,\pi]} |\mathrm{Im}(L+\tilde{L})(t,sre^{\sqrt{-1}\theta})|^2\mathrm{Re}\Big(a\frac{1-re^{\sqrt{-1}\theta}}{1+re^{\sqrt{-1}\theta}}+b\frac{1+re^{\sqrt{-1}\theta}}{1-re^{\sqrt{-1}\theta}}\Big)\,\frac{d\theta}{2\pi}\Big]\,dt\Big\}^{1/2}& \\
\times M_{t_1}^{3/4}\,\sqrt{C_3(t_2-t_1)/2\pi}&   
\end{align*}
with a universal constant $C_3>0$ (that comes from M.~Riesz's theorem) and moreover that this converges to $0$ as $r\nearrow1$ thanks to \cite[p.7--8]{Koosis:Book}, \eqref{Eq-4.7}, the continuity of $\mathrm{Im}(L+\tilde{L})(t,\zeta)$ in $(t,\zeta)$ and $\mathrm{Im}(L+\tilde{L})(t,\pm s) = 0$ (due to $\hat{h}(-\theta) = \hat{h}(\theta)$). By the subordination relation in Proposition \ref{P-3.1} with Littlewood's subordination principle again 
\begin{equation}\label{Eq-4.8}
\begin{aligned}
\Big\Vert\mathrm{Im}\Big(2L(t,re^{\sqrt{-1}(-)}) + a\frac{1-re^{\sqrt{-1}(-)}}{1+re^{\sqrt{-1}(-)}} + b\frac{1+re^{\sqrt{-1}(-)}}{1-re^{\sqrt{-1}(-)}}\Big)\,2\mathrm{Re}L(t,re^{\sqrt{-1}(-)})\Big\Vert_{3/2}& \\
\leq 4\Vert H(t,re^{\sqrt{-1}(-)})\Vert_{3/2} \leq 4M_{t_1}&
\end{aligned}
\end{equation} 
for every $t \geq t_1$ and $0 \leq r < 1$, and we can easily confirm, with the help of facts in \cite[p.9; p.88--89]{Koosis:Book}, that the first term of the right-hand side of \eqref{Eq-4.6} converges to 
\begin{align*} 
&\frac{1}{2}\int_{t_1}^{t_2}\int_{(-\pi,\pi]}\mathrm{Im}\left(2L(t,se^{\sqrt{-1}\theta})+a\frac{1-se^{\sqrt{-1}\theta}}{1+se^{\sqrt{-1}\theta}}+b\frac{1+se^{\sqrt{-1}\theta}}{1-se^{\sqrt{-1}\theta}}\right) \\
&\quad\quad\quad\quad\quad\quad\times\left(2\pi H(2\hat{h}(t,-))(\theta) - a \tan(\theta/2) + b\cot(\theta/2)\right)\,2\hat{h}(t,\theta)\,d\theta\,dt  
\end{align*} 
as $r\nearrow1$. Consequently, letting $Z(s) := Z_{\tau(P),\tau(Q)} -\frac{1-(a+b)^2}{4}\log s$ we have 
\begin{equation}\label{Eq-4.9} 
\begin{aligned}
&-\Big\{2\int\int_{(-\pi,\pi]^2}\log|1-se^{\sqrt{-1}(\alpha-\beta)}|\,(\hat{\mu}_{t_1} + a\delta_{\pi}+\delta_0)(d\alpha)\,\hat{\mu}_{t_1}(d\beta) + Z(s) \Big\}\\
&=\frac{1}{2}\int_{t_1}^{t_2}\int_{(-\pi,\pi]}\mathrm{Im}\left(2L(t,se^{\sqrt{-1}\theta})+a\frac{1-se^{\sqrt{-1}\theta}}{1+se^{\sqrt{-1}\theta}}+b\frac{1+se^{\sqrt{-1}\theta}}{1-se^{\sqrt{-1}\theta}}\right) \\
&\quad\quad\quad\quad\quad\quad\times\left(2\pi H(2\hat{h}(t,-))(\theta) - a \tan(\theta/2) + b\cot(\theta/2)\right)\,2\hat{h}(t,\theta)\,d\theta\,dt \\
&\quad-\Big\{2\int\int_{(-\pi,\pi]^2} \log|1-se^{\sqrt{-1}(\alpha-\beta)}|\,(\hat{\mu}_{t_2} + a\delta_{\pi}+\delta_0)(d\alpha)\,\hat{\mu}_{t_2}(d\beta) + Z(s)\Big\}.
\end{aligned}
\end{equation} 
Write $k(t,\theta) := 2\pi H(2\hat{h}(t,-))(\theta) - a \tan(\theta/2) + b\cot(\theta/2)$ for simplicity. By \eqref{Eq-4.8} with \cite[p.9; p.88--89]{Koosis:Book} one has 
\begin{equation}\label{Eq-4.10} 
\Vert k(t,-)\,2\hat{h}(t,-)\Vert_{3/2} \leq 2M_{t_1}/\pi
\end{equation} 
for every $t \geq t_1$. By Lemma \ref{L-4.1} with the aid of \eqref{Eq-4.7} and \eqref{Eq-4.10}, and moreover by \cite[Proposition 10.11 (a)]{Voiculescu:AdvMath99}
$$
\int_{t_1}^{t_2}\int_{(-\pi,\pi]} k(t,\theta)^2\, 2\hat{h}(t,\theta)\,d\theta\,dt = \int_{t_1}^{t_2} \varphi^*(U_t P U_t^*:Q)\,dt \leq 2i^*(U_{t_1}PU_{t_1}^*:Q) < +\infty.
$$ 
By the H\"{o}lder inequality, \eqref{Eq-4.7}, \eqref{Eq-4.10} and M.~Riesz's theorem   
$$
\int_{t_1}^{t_2}\int_{(-\pi,\pi]} |2\pi H(2\hat{h}(t,-))(\theta)|\,|k(t,\theta)|\,2\hat{h}(t,\theta)\,d\theta\,dt  < +\infty.
$$
Since 
\begin{align*} 
a\Big|\mathrm{Im}\Big(\frac{1- se^{\sqrt{-1}\theta}}{1+ se^{\sqrt{-1}\theta}}\Big)\Big| &+ b\Big|\mathrm{Im}\Big(\frac{1+ se^{\sqrt{-1}\theta}}{1-se^{\sqrt{-1}\theta}}\Big)\Big| 
\leq 
a|\tan(\theta/2)| + b|\cot(\theta/2)| \\
&\leq 
2a\tan(\pi/4) + 2b\cot(\pi/4) + k(t,\theta) + |2\pi H(2\hat{h}(t,-))(\theta)|,
\end{align*} 
we easily see, by \cite[p.9; p.88--89]{Koosis:Book} again,  that the first term of the right-hand side of \eqref{Eq-4.9} converges to $\frac{1}{2}\int_{t_1}^{t_2} \varphi^*(U_t P U_t^*:Q)\,dt$ as $s\nearrow1$. By Lemma \ref{L-4.2} with the aid of the first 5 lines of \cite[p.147]{Voiculescu:AdvMath99} we finally get 
$$
-\chi_\mathrm{orb}(U_{t_2}PU_{t_2}^*) + \frac{1}{2}\int_{t_1}^{t_2} \varphi^*(U_t P U_t^*:Q)\,dt = -\chi_\mathrm{orb}(U_{t_1}PU_{t_1}^*,Q).
$$

By \cite[Theorem 2.1]{HiaiUeda:AIHP09}, \cite[Proposition 10.11 (a)]{Voiculescu:AdvMath99} and \cite[Proposition 4.6]{HiaiMiyamotoUeda:IJM09} one has 
$$
\int_{t_1}^{+\infty} -\chi_\mathrm{orb}(U_{t_2} P U_{t_2}^*,Q)\,dt_2 \leq \int_{t_1}^{+\infty} \varphi^*(U_{t_2}PU_{t_2}^*:Q)\,dt_2 = 2\,i^*(U_{t_1}PU_{t_1}^*:Q) < +\infty,
$$ 
implying $\lim_{t_2\nearrow+\infty}\chi_\mathrm{orb}(U_{t_2}PU_{t_2}^*,Q) = 0$. (This trick originates in a preprint version of \cite{HiaiUeda:AIHP09}.) By \cite[Proposition 2.5 (4), Proposition 4.6]{HiaiMiyamotoUeda:IJM09} one has $\lim_{t_1\searrow0}\chi_\mathrm{orb}(U_{t_1}PU_{t_1}^*,Q) = \chi_\mathrm{orb}(P,Q)$. Hence we are done.    
\end{proof} 

By \eqref{Eq-3.2} $H(t,\zeta)$ becomes the constant $(1-(a+b)^2)/4$ in the time stationary case; hence the assumption of Lemma \ref{L-4.4} is not strange. Here is a sample of application of Lemma \ref{L-4.4}.   

\begin{corollary}\label{C-4.5} Assume that the measure $\mu_0$ {\rm(}see \S3{\rm)} has an $L^3$-density with respect to $x(1-x)\,dx$ on $[0,1]$ and is supported in $[\alpha,\beta]$ such that $\alpha \gneqq 0$ if $\tau(P)\neq\tau(Q)$ and $\beta \lneqq 1$ if $\tau(P)+\tau(Q)\neq1$. Then $i^*(\mathbb{C}P+\mathbb{C}(I-P)\,;\mathbb{C}Q+\mathbb{C}(I-Q))=-\chi_\mathrm{orb}(P,Q)$ holds. 
\end{corollary}
\begin{proof} For simplicity, assume both $a=|\tau(P)-\tau(Q)|\neq0$ and $b=|\tau(P)+\tau(Q)-1|\neq0$. It is easy to see that $\hat{\mu}_0$ has an $L^3$-density $\hat{h}(0,\theta)$ (with respect to $d\theta$); hence $L(0,\zeta)$ is a function  in $\zeta$ of Hardy class with exponent $3$ by M.~Riesz's theorem with a standard fact (see \cite[p.9; p.88--89]{Koosis:Book}). Moreover, the assumption here implies that $L(0,\zeta)$ has analytic continuation across both $\zeta=\pm1$. Since $\lim_{\zeta\to\pm1} L(0,\zeta) = 0$, $L(0,\zeta)$ admits a power series expansion without constant term around $\zeta=\pm1$. Thus $H(0,\zeta)$ is bounded in some neighborhoods at both $\zeta=\pm1$. It is plain to show that $H(0,\zeta)$ is a function in $\zeta$ of Hardy class with exponent $3/2$. Hence the assertion follows thanks to the subordination relation in Proposition \ref{P-3.1} with Littlewood's subordination principle (see \cite[Theorem 1.7]{Duren:Book}).
\end{proof} 

The above fact suggests that the question should be affirmative without assuming $\tau(P)=\tau(Q)=1/2$. Only missing piece in our attempt is apparently a more detailed study of $H(t,\zeta)$ and/or the conformal transformations $f_t(\zeta)$; thus the question comes down to a study of Loewner--Kufarev equations.  
    
\section*{Acknowledgement} We thank Fumio Hiai for discussions on this subject matter and comments to a draft of the present notes.  
}

\end{document}